\documentclass{amsart}
\usepackage{amsmath,amssymb,amsthm}

\usepackage[dvipdfmx]{graphicx}
\graphicspath{{./picture/}}

\usepackage{color}
\usepackage{here}
\usepackage{url}
\usepackage{hyperref}

\theoremstyle{plain}
 \newtheorem{thm}{Theorem}[section]
 \newtheorem{lem}[thm]{Lemma}
 \newtheorem{prop}[thm]{Proposition}
 \newtheorem{cor}[thm]{Corollary}

 \newtheorem{ques}[thm]{Question}
\theoremstyle{definition}
 
\theoremstyle{remark}
 \newtheorem{rem}[thm]{Remark}

\newcommand{\vol}{\mathrm{vol}}
\newcommand{\length}{\mathrm{length}}

\begin{document}
\bibliographystyle{siam}

\title[Quotient topology on the set of commensurability classes]{Quotient topology on the set of commensurability classes of hyperbolic 3-manifolds}
\author{Ken'ichi YOSHIDA}
\address{Department of Mathematics, Saitama University, 255 Shimo-Okubo, Sakura-ku, Saitama-shi, Saitama 338-8570, Japan}
\email{kyoshida@mail.saitama-u.ac.jp}
\date{}

\begin{abstract}
We investigate relation between Dehn fillings and commensurability of hyperbolic 3-manifolds. 
The set consisting of the commensurability classes of hyperbolic 3-manifolds 
admits the quotient topology induced by the geometric topology. 
We show that this quotient space satisfies some separation axioms. 
Roughly speaking, this means that commensurability classes are sparsely distributed 
in the space consisting of the hyperbolic 3-manifolds. 
\end{abstract}

\maketitle

\section{Introduction}
\label{section:intro}

Let $\mathcal{H}$ denote 
the set consisting of the finite volume oriented hyperbolic 3-manifolds. 
For simplicity, we will often omit ``finite volume oriented''. 
According to J\o rgensen and Thurston, 
the set $\mathcal{H}$ admits a topology 
induced by the Gromov-Hausdorff convergence. 
This topology is called the \textit{geometric topology}. 
The space $\mathcal{H}$ is homeomorphic to 
the ordinal $\omega^\omega$ with the order topology 
by the following manner. 
Let $M$ be a hyperbolic 3-manifold with $k$ cusps. 
Fix meridians and longitudes on the cusps of $M$. 
For $r=(r_{1}, \dots , r_{k}) \in (\mathbb{Q} \cup \{\infty\} \cup \{\emptyset\})^{k}$, 
let $M(r)$ be a 3-manifold obtained by gluing solid tori to $M$ along the cusps, 
so that the boundary of compressing disk in the $i$-th solid tori has the slope $r_{i}$. 
If $r_{i} = \emptyset$, the $i$-th cusp is not filled. 
Then we simply say that $M(r)$ is the \textit{Dehn filling} of $M$ along the slope $r$. 
Moreover, it holds that $\vol (M(r)) < \vol (M)$ unless $r$ is empty. 
The volume function on $\mathcal{H}$ is continuous and proper. 
A closed hyperbolic 3-manifold is an isolated point in $\mathcal{H}$. 
On the other hand, a cusped hyperbolic 3-manifold is a limit point in $\mathcal{H}$. 
Its neighborhood contains Dehn fillings along sufficiently long slopes. 
For long slopes $r$, 
the cores in the filled solid tori of $M(r)$ are closed geodesics of short lengths. 
We refer the reader to 
\cite{benedetti1992lectures, gromov1981hyperbolic, ratcliffe2006foundations} 
for more details.

Two oriented manifolds are \textit{commensurable} 
if they have a common finite covering 
such that the covering maps preserves the orientations. 
The fundamental group $\pi_{1}(M)$ of $M \in \mathcal{H}$ is residually finite, 
since it is linear. 
Hence there are finite coverings of $M$ of arbitrarily large degree. 
The commensurability is an equivalent relation. 
Let $\mathcal{C}$ denote 
the set consisting of the commensurability classes in $\mathcal{H}$. 
The geometric topology on $\mathcal{H}$ induces the quotient topology on $\mathcal{C}$. 
Of course, a closed manifold and a cusped manifold are not commensurable. 
Independently of closedness, there are two kinds of elements in $\mathcal{C}$: 
arithmetic and non-arithmetic. 
We refer the reader to \cite{maclachlan2013arithmetic} for details. 
Theorem~\ref{thm:arithvol} shows that 
the arithmetic hyperbolic 3-manifolds are comparatively rare. 
It follows from a volume formula for arithmetic hyperbolic 3-manifolds. 

\begin{thm}[\cite{maclachlan2013arithmetic} Theorem 11.2.1]
\label{thm:arithvol}
For any $v>0$, there are at most finitely many arithmetic hyperbolic 3-manifolds 
with volume $<v$. 
\end{thm}

Consequently, 
for any cusped hyperbolic 3-manifold $M$, 
there are at most finitely many arithmetic Dehn fillings of $M$. 
Thus we may say that ``most of the hyperbolic 3-manifolds are non-arithmetic.''
Every non-arithmetic class has a minimal element 
in the category of hyperbolic 3-orbifolds 
(see also \cite{borel1981commensurability}). 

\begin{thm}[Margulis~\cite{margulis1991discrete}]
\label{thm:nonarithcomm}
For a non-arithmetic class $C \in \mathcal{C}$, 
there is an oriented hyperbolic 3-orbifold $M_{C}$ such that 
any $M \in C$ is a covering of $M_{C}$. 
\end{thm}

H. Yoshida~\cite{yoshida2017commensurability} showed that 
every cusped hyperbolic 3-manifold has infinitely many Dehn fillings 
which are mutually incommensurable, 
using Theorems~\ref{thm:arithvol} and \ref{thm:nonarithcomm}. 
We investigate more precise relation between Dehn fillings and commensurability 
by considering the quotient topology of $\mathcal{C}$.

The Ehrenpreis conjecture, proved by Kahn and Markovic~\cite{kahn2015good}, 
states that any two closed Riemann surfaces of genus $\geq 2$ 
have finite coverings which are arbitrarily close with respect to the Teichm\"{u}ller metric. 
Francaviglia, Frigerio, and Martelli~\cite[Remark 6.8]{francaviglia2012stable} 
proposed to generalize this conjecture to questions in any dimension. 
Specifically, let $\mathcal{A}$ be a set of manifolds 
closed under commensurability, 
and suppose that $\mathcal{A}$ is equipped with a metric. 
Then one may ask whether two manifolds in $\mathcal{A}$ 
have finite coverings which are arbitrarily close. 
They called such a question an \textit{Ehrenpreis problem}. 

We will give an ad hoc metric on $\mathcal{H}$ in Section~\ref{section:metric}. 
At present we do not solve the Ehrenpreis problem for this metric. 
Nevertheless, Theorem~\ref{thm:main} supports the negative answer, 
and it answers \cite[Question 8.6]{ohtsuki2015problems}.

\begin{thm}
\label{thm:main}
The quotient space $\mathcal{C}$ is normal Hausdorff. 
\end{thm}

\begin{cor}
\label{cor:vol}
Suppose that $M_{0}, M_{1} \in \mathcal{H}$ are not commensurable. 
Then there is a continuous function $f \colon \mathcal{H} \to \mathbb{R}_{>0}$ 
satisfying the following properties: 
\begin{itemize}
\item If $M,N \in \mathcal{H}$ and $N$ is an $n$-sheeted covering of $M$,  
then $f(N) = n f(M)$. 
\item $\dfrac{f(M_{0})}{\vol(M_{0})} \neq \dfrac{f(M_{1})}{\vol(M_{1})}$. 
\end{itemize}
\end{cor}
\begin{proof}
Let $\pi \colon \mathcal{H} \to \mathcal{C}$ 
denote the natural projection. 
By Urysohn's lemma, 
there is a continuous function $\bar{f} \colon \mathcal{C} \to \mathbb{R}_{>0}$ 
such that $\bar{f}(\pi(M_{0})) \neq \bar{f}(\pi(M_{1}))$. 
It is sufficient to set $f = (\bar{f} \circ \pi) \cdot \vol$. 
\end{proof}

We ask whether commensurability and Dehn fillings characterize the volume function. 

\begin{ques}
\label{ques:vol}
Is there $f$ in Corollary~\ref{cor:vol} satisfying the following additional property? 
\begin{itemize}
\item If $M,N \in \mathcal{H}$ and $N$ is a Dehn filling of $M$, then $f(N) < f(M)$. 
\end{itemize}
\end{ques}
The values of $f$ for closed classes can be made smaller. 
Hence the assertion holds unless both $M_{0}$ and $M_{1}$ are cusped.

\begin{rem}
\label{rem:orb}
We may consider the space $\mathcal{H}_{\mathrm{orb}}$ 
consisting of the finite volume hyperbolic 3-orbifolds 
with respect to the geometric topology. 
Let $\mathcal{C}_{\mathrm{orb}}$ denote the quotient space of $\mathcal{H}_{\mathrm{orb}}$ 
up to commensurability. 
Although $\mathcal{C}_{\mathrm{orb}}$ and $\mathcal{C}$ are identified as sets 
by Selberg's lemma~\cite[\S 7.6]{ratcliffe2006foundations}, 
their quotient topologies are different. 
We will prove it in Proposition~\ref{prop:orb}. 
Nevertheless, we can show the separability of $\mathcal{C}_{\mathrm{orb}}$ 
in the same manner as for $\mathcal{C}$. 
\end{rem}

\section{A metric on the space $\mathcal{H}$}
\label{section:metric} 

We define an ad hoc metric on $\mathcal{H}$. 
We construct a metric graph whose vertex set is $\mathcal{H}$. 
%We first connect each pair of the vertices with an edge of length one. 
Let $M,N \in \mathcal{H}$ and let $N$ be a Dehn filling of $M$. 
Then we connect $M$ and $N$ with an edge of length $\vol (M) - \vol (N)$. 
After connecting edges for every such manifolds, 
we obtain a metric graph. 
For $M_{0}, M_{1} \in \mathcal{H}$, 
there is $N \in \mathcal{H}$ 
such that $M_{0}$ and $M_{1}$ are Dehn fillings of $N$ 
by the Lickorish-Wallace theorem. 
Hence any two vertices of this graph are connected by two edges. 
In particular, this graph is connected. 
Let $d$ denote the path metric of this graph. 
Clearly $d$ is a pseudometric.

\begin{prop}
\label{prop:metric}
The pseudometric $d$ on $\mathcal{H}$ has the following desired properties: 
\begin{itemize}
\item 
$d$ is a metric. 
\item 
The induced topology on $\mathcal{H}$ is the geometric topology. 
\item 
The metric space $(\mathcal{H}, d)$ is complete. 
\item 
The volume function on $\mathcal{H}$ is uniformly continuous. 
\item 
If $M_{0},M_{1},M_{2} \in \mathcal{H}$, 
$M_{1}$ is a Dehn filling of $M_{0}$, 
and $M_{2}$ is a Dehn filling of $M_{1}$, 
then $d(M_{0}, M_{1}) < d(M_{0}, M_{2})$. 
\end{itemize}
\end{prop}

We need the following lemma for the first and second assertions. 

\begin{lem}
\label{lem:nbd}
For $v>0$, 
let $M_{1}, \dots , M_{k} \in \mathcal{H}$ denote the hyperbolic 3-manifolds of volume $v$. 
Then there exists $\delta = \delta (v) > 0$ satisfying the following condition: 
If $P \in \mathcal{H}$ satisfies $|\vol (P) - v| < \delta$, 
then $P$ is a (possibly empty) Dehn filling of only one of $M_{1}, \dots , M_{k}$. 
\end{lem}
\begin{proof}
For $v>0$, 
there is a finite subset $A$ of $\mathcal{H}$ 
such that any hyperbolic 3-manifold of volume at most $v$ 
is a (possibly empty) Dehn filling of a manifold in $A$ 
(see \cite[Theorem E.4.8]{benedetti1992lectures}). 
Let $0 < \delta^{\prime} < v - \max (\vol (A) \setminus \{v\})$. 
Then $P \in \mathcal{H}$ satisfying $v - \delta^{\prime} < \vol (P) \leq v$ 
is a (possibly empty) Dehn filling of one of $M_{1}, \dots , M_{k}$. 
There is not a decreasing convergent sequence in $\vol(\mathcal{H})$. 
Hence we can take $\delta^{\prime} > 0$ 
so that $P \in \mathcal{H}$ satisfying $|\vol (P) - v| < \delta^{\prime}$ 
is a (possibly empty) Dehn filling of one of $M_{1}, \dots , M_{k}$. 

Assume that there does not exist $\delta > 0$ satisfying the asserted condition. 
Then there is a sequence $\{ P_{n} \}$ in $\mathcal{H}$ 
such that $\lim_{n \to \infty} \vol(P_{n}) = v$, 
and the manifolds $P_{n}$ are Dehn fillings of both $M_{i}$ and $M_{j}$ 
for distinct $1 \leq i,j \leq k$. 
Hence the sequence $\{ P_{n} \}$ converges to both $M_{i}$ and $M_{j}$ 
with respect to the geometric topology. 
This contradicts the fact that the geometric topology is Hausdorff. 
\end{proof}

\begin{proof}[Proof of Proposition~\ref{prop:metric}]	
The construction of the metric graph implies that  
\[
d(M,N) \geq |\vol (M) - \vol (N)|. 
\]
The fourth assertion follows from this inequality. 
	
For the fifth assertion, 
suppose that  
$M_{0},M_{1},M_{2} \in \mathcal{H}$, 
$M_{1}$ is a Dehn filling of $M_{0}$
and $M_{2}$ is a Dehn filling of $M_{1}$. 
Then 
$d(M_{0}, M_{1}) = \vol (M_{0}) - \vol (M_{1})$ 
and $d(M_{0}, M_{2}) = \vol (M_{0}) - \vol (M_{2})$. 
Hence $d(M_{0}, M_{1}) < d(M_{0}, M_{2})$. 

For the first assertion, 
let $M$ and $N$ be distinct elements of $\mathcal{H}$. 
If $\vol (M) \neq \vol (N)$, 
then $d(M,N) \geq |\vol (M) - \vol (N)| > 0$. 
Suppose that $\vol (M) = \vol (N) = v$. 
Let $\delta$ be the constant given in Lemma~\ref{lem:nbd}. 
Then any path joining $M$ and $N$ in the metric graph 
contains a vertex $P$ such that $|\vol (P) - v| \geq \delta$. 
Hence $d(M,N) \geq 2\delta > 0$. 

For the second assertion, 
it is sufficient to show that 
the convergences with respect to the two topologies are equivalent. 
Consider a sequence $\{ M_{n} \}$ in $\mathcal{H}$. 
Suppose that $\{ M_{n} \}$ converges to $M$ 
with respect to the geometric topology. 
Then $M_{n}$ is a Dehn filling of $M$ for sufficiently large $n$. 
Hence $\lim_{n \to \infty} d(M_{n}, M) = 0$ by the construction. 
Conversely, 
suppose that $\lim_{n \to \infty} d(M_{n}, M) = 0$. 
Then $\lim_{n \to \infty} |\vol (M_{n})- \vol (M)| = 0$. 
By Lemma~\ref{lem:nbd}, 
$M_{n}$ is a Dehn filling of $M$ 
for sufficiently large $n$. 
Hence the sequence  $\{ M_{n} \}$ converges to $M$ 
with respect to the geometric topology. 

For the third assertion, 
take a Cauchy sequence $\{ M_{n} \}$ in $(\mathcal{H}, d)$. 
The inequality $|\vol (M) - \vol (N)| \leq d(M,N)$ implies that 
$\{ \vol (M_{n}) \}$ is a Cauchy sequence in $\mathbb{R}$. 
In particular, the volumes $\vol (M_{n})$ are bounded. 
Hence there is a convergent subsequence of $\{ M_{n} \}$. 
Then the Cauchy sequence $\{ M_{n} \}$ is convergent in $\mathcal{H}$. 
\end{proof}

\section{Separability of the quotient space $\mathcal{C}$}
\label{section:main} 

Theorem~\ref{thm:main} follows from 
Proposition~\ref{prop:t1} and Theorem~\ref{thm:0dim}. 

\begin{prop}
\label{prop:t1}
The quotient space $\mathcal{C}$ is $T_{1}$-space, 
i.e. every singleton in $\mathcal{C}$ is closed. 
\end{prop}

\begin{thm}
\label{thm:0dim}
The quotient space $\mathcal{C}$ is zero-dimensional 
with respect to the small inductive dimension, 
i.e. every point in $\mathcal{C}$ has a neighborhood base 
consisting of clopen sets. 
\end{thm}

Theorem~\ref{thm:0dim} implies that $\mathcal{C}$ is regular, 
i.e. every point in $\mathcal{C}$ has a neighborhood base 
consisting of closed sets.  
Theorem~\ref{thm:main} follows 
from the fact that a regular Lindel\"{o}f space is normal 
(see \cite[Theorem 16.8]{willard2004general}). 
Since $\mathcal{C}$ is a countable set, 
trivially it is Lindel\"{o}f. 

\begin{proof}[Proof of Proposition~\ref{prop:t1}]
Take a class $C \in \mathcal{C}$. 
Let $\pi \colon \mathcal{H} \to \mathcal{C}$ 
denote the natural projection. 
It is sufficient to show that 
$\pi^{-1}(C)$ is closed in $\mathcal{H}$. 
This follows from the fact that 
the set $\vol (C) = \{ \vol (M) \mid M \in C \}$ 
is discrete in $\mathbb{R}$. 
If $C$ is arithmetic, 
then $\vol (C)$ is discrete by Theorem~\ref{thm:arithvol}. 
Suppose that $C$ is non-arithmetic. 
By Theorem~\ref{thm:nonarithcomm}, 
there is a hyperbolic orbifold $M_{C}$ 
such that any $M \in C$ is a finite covering of $M_{C}$. 
Hence $\vol (C)$ is contained in $\mathbb{Z}_{>0} \cdot \vol (M_{C})$. 
\end{proof}

%Remark: closed geodesics in hyperbolic orbifolds 

For $M \in \mathcal{H}$ and $\epsilon >0$, 
let $B(M; \epsilon)$ denote 
the open ball of radius $\epsilon$ about $M$ 
with respect to the metric constructed in Section~\ref{section:metric}. 
Let $M(r)$ denote the Dehn filling of $M$ 
along a slope $r=(r_{1}, \dots , r_{k})$. 
For $1 \leq i \leq k$, 
let $\gamma_{i}$ denote the $i$-th filled core in $M(r)$, i.e. 
the simple closed geodesic in the $i$-th filled solid torus in $M(r)$. 
We identify $M(r) \setminus \bigcup \gamma_{i}$ with $M$.

\begin{lem}
\label{lem:cover}
For $M \in \mathcal{H}$, 
there exists $\epsilon (M) >0$ satisfying the following conditions: 
At first, any element of $B(M; \epsilon (M))$ is a (possibly empty) Dehn filling of $M$. 
Suppose that $M(r) \in B(M; \epsilon (M))$, 
and $f \colon M(r) \to N$ is a finite covering to a hyperbolic orbifold $N$. 
Then $f$ can be restricted a finite covering 
from $M$ to an orbifold. 
Furthermore, the closed geodesics in $N$ 
coming from the filled cores in $M(r)$ 
are shorter than any other closed geodesic in $N$. 
\end{lem}
\begin{proof}
The first condition follows from Lemma~\ref{lem:nbd}. 
Let $n$ denote the degree of $f$. 
There is a hyperbolic 3-orbifold of smallest volume $v_{0}$, 
which is determined by Marshall and Martin~\cite{marshall2012minimal}. 
Hence $n = \vol (M(r))/\vol (N) < \vol(M)/v_{0}$. 

Suppose that $\epsilon (M) >0$ is sufficiently small. 
Then the lengths of $\gamma_{i}$'s are small. 
The representation of $\pi_{1}(M(r) \setminus \bigcup \gamma_{i})$ 
induced by the holonomy representation of $\pi_{1}(M(r))$ 
is algebraically near to the holonomy representation of $\pi_{1}(M)$. 
Moreover, $n \cdot \length (\gamma_{i})$ is smaller than 
the length of a closed geodesic in $M(r)$ 
which is not contained in the filled solid tori. 
Therefore $f^{-1}(f(\bigcup \gamma_{i})) = \bigcup \gamma_{i}$. 
Then $f$ can be restricted to 
a covering from $M = M(r) \setminus \bigcup \gamma_{i}$ 
to $N \setminus f(\bigcup \gamma_{i})$. 
\end{proof}

\begin{lem}
\label{lem:t2}
For $M \in \mathcal{H}$, 
let $\epsilon (M)$ be as in Lemma~\ref{lem:cover}. 
Suppose that 
$P, Q \in \mathcal{H}$, 
$\epsilon < \epsilon (P)$, $\epsilon^{\prime} < \epsilon (Q)$, 
and there are $R \in B(P; \epsilon)$ 
and $R^{\prime} \in B(Q; \epsilon^{\prime})$ 
such that $R$ and $R^{\prime}$ are commensurable. 
Then either there is $Q^{\prime} \in B(P; \epsilon)$ commensurable to $Q$ 
or there is $P^{\prime} \in B(Q; \epsilon^{\prime})$ commensurable to $P$. 
Furthermore, if $R = R^{\prime}$, 
either $Q \in B(P; \epsilon)$ or $P \in B(Q; \epsilon^{\prime})$. 
\end{lem}
\begin{proof}
We may assume that $R$ and $R^{\prime}$ are non-arithmetic by Theorem~\ref{thm:arithvol}. 
Hence there is an orbifold $\hat{R}$ such that 
$R$ and $R^{\prime}$ are finite coverings of $\hat{R}$ by Theorem~\ref{thm:nonarithcomm}. 
We may regard $P \subset R$ and $Q \subset R^{\prime}$ 
so that $\hat{P}, \hat{Q} \subset \hat{R}$ are respectively 
the image of $P$ and $Q$ by the coverings. 
Then $P$ and $Q$ are respectively 
finite coverings of $\hat{P}$ and $\hat{Q}$ 
by Lemma~\ref{lem:cover}. 
The second assertion in Lemma~\ref{lem:cover} 
implies that $\hat{P} \subset \hat{Q}$ or $\hat{Q} \subset \hat{P}$. 
If $\hat{P} \subset \hat{Q}$, 
the preimage $Q^{\prime}$ of $\hat{Q}$ by the covering $R \to \hat{R}$ 
satisfies $P \subset Q^{\prime} \subset R$. 
Then $Q^{\prime}$ is commensurable to $Q$, 
and $Q^{\prime} \in B(P; \epsilon)$. 
The symmetric argument holds if $\hat{Q} \subset \hat{P}$. 
If $R = R^{\prime}$, we may assume that $P = \hat{P}$ and $Q = \hat{Q}$. 
\end{proof}

Proposition~\ref{prop:t1} and Lemma~\ref{lem:t2} immediately imply that 
$\mathcal{C}$ is Hausdorff.

\begin{proof}[Proof of Theorem~\ref{thm:0dim}]
If a class $C \in \mathcal{C}$ consists of closed manifolds, 
then $\pi^{-1}(C)$ is a clopen subset in $\mathcal{H}$. 
Hence $C$ is an isolated point in $\mathcal{C}$. 
Then the assertion for $C$ is trivial. 

Suppose that 
a class $C \in \mathcal{C}$ consists of cusped manifolds. 
We arbitrarily arrange the countable set so that 
$\pi^{-1}(C) = \{ M_{0,j} \in \mathcal{H} \mid j \in \mathbb{Z}_{\geq 0} \}$. 
Take $\epsilon_{i,j} > 0$ for $i, j \geq 0$. 
We write $U_{0,j} = B(M_{0,j}; \epsilon_{0,j})$ 
and $U_{0} = \bigcup_{j} U_{0,j}$. 
Inductively, for $i \geq 1$, 
we arrange the set so that 
$\pi^{-1}(\pi(U_{i-1})) = \{ M_{i,j} \in \mathcal{H} \mid j \in \mathbb{Z}_{\geq 0} \}$. 
We write $U_{i,j} = B(M_{i,j}; \epsilon_{i,j})$ 
and $U_{i} = U_{i-1} \cup \bigcup_{j} U_{i,j}$. 
After all, 
we have an increasing sequence 
$U_{0} \subset U_{1} \subset U_{2} \subset \dots \subset \mathcal{H}$. 
We write $U = U(\{ \epsilon_{i,j} \}) = \bigcup_{i} U_{i}$. 
The construction implies that 
$\pi^{-1}(\pi(U(\{ \epsilon_{i,j} \}))) = U(\{ \epsilon_{i,j} \})$. 
The set $\{ \pi(U(\{ \epsilon_{i,j} \}) \mid \epsilon_{i,j}>0 \}$ 
is a neighborhood base of $C$.

The set $U$ is open in $\mathcal{C}$, since it is the union of open balls. 
We claim that  
$U$ is closed in $\mathcal{C}$ if $\epsilon_{i,j}$ are sufficiently small. 
We take $0 < \epsilon_{i,j} < \epsilon (M_{i,j})$. 
Suppose that a sequence $\{ N_{k} \}$ in $U$ converges to $N \in \mathcal{H}$. 
It is sufficient to show that $N \in U$. 
Fix $\epsilon < \epsilon (N)$, 
where $\epsilon (N)$ is the constant given in Lemma~\ref{lem:cover}. 
By Theorem~\ref{thm:arithvol}, 
we may assume that $N_{k}$ is non-arithmetic, 
and $N_{k} \in B(N; \epsilon) \setminus \pi^{-1}(C)$ for all $k$. 

For fixed $n$, suppose that $N_{n} \in U_{i} \setminus U_{i-1}$. 
(Let $U_{-1} = \pi^{-1}(C)$.) 
We first suppose that $N_{n} \notin \pi^{-1}(\pi(U_{i-1}))$. 
Then there is $M_{i,j} \in \pi^{-1}(\pi(U_{i-1}))$ 
and $N_{n} \in B(M_{i,j}; \epsilon_{i,j})$. 
Since $N_{n} \in B(M_{i,j}; \epsilon_{i,j})$ 
and $N_{n} \in B(N; \epsilon)$, 
we have either $N \in B(M_{i,j}; \epsilon_{i,j})$ 
or $M_{i,j} \in B(N; \epsilon)$ by Lemma~\ref{lem:t2}. 
If $N \in B(M_{i,j}; \epsilon_{i,j})$, 
we have $N \in U$. 
If $M_{i,j} \in B(N; \epsilon)$, 
the proof is reduced to the case that $N_{n} \in \pi^{-1}(\pi(U_{i-1}))$. 

We next suppose that $N_{n} \in \pi^{-1}(\pi(U_{i-1}))$. 
There are $N_{n}^{\prime} \in U_{i-1}$ 
and $M_{i-1,j} \in \pi^{-1}(\pi(U_{i-2}))$ 
such that $N_{n}^{\prime}$ is commensurable to $N_{n}$, 
and $N_{n}^{\prime} \in B(M_{i-1,j}; \epsilon_{i-1,j})$. 
Since $N_{n}^{\prime} \in B(M_{i-1,j}; \epsilon_{i-1,j})$ 
and $N_{n} \in B(N; \epsilon)$, 
either there is 
$N^{\prime} \in B(M_{i-1,j}; \epsilon_{i-1,j})$ 
commensurable to $N$ 
or there is $M_{i-1,j}^{\prime} \in B(N; \epsilon)$ 
commensurable to $M_{i-1,j}$ 
by Lemma~\ref{lem:t2}. 
If $N^{\prime} \in B(M_{i-1,j}; \epsilon_{i-1,j})$, 
we have $N \in U$. 
If $M_{i-1,j}^{\prime} \in B(N; \epsilon)$, 
we can continue the argument for $M_{i-1,j}^{\prime} \in \pi^{-1}(\pi(U_{i-2}))$ 
instead of $N_{n} \in \pi^{-1}(\pi(U_{i-1}))$. 
Then we have eventually $N \in U$. 
\end{proof}

We show that 
the quotient topology is too fine to metrize.

\begin{prop}
\label{prop:n1c}
The quotient space $\mathcal{C}$ is not first countable. 
In particular, $\mathcal{C}$ is not metrizable. 
\end{prop}
\begin{proof}
The Frech\'{e}t-Urysohn fan $F$ is defined 
as the quotient topological space $(\bigsqcup_{\mathbb{Z}_{>0}} X) / \sim $, 
where $X = \{ 0 \} \cup \{ 1/n \mid n \in \mathbb{Z}_{>0}\} 
\subset \mathbb{R}$ 
and all the copies of $0 \in X$ are collapsed to a point. 
Then $F$ is not first countable  
(see \cite[c-14]{hart2003encyclopedia} for details). 

Let $C \in \mathcal{C}$ be a class consisting of cusped manifolds. 
Then a neighborhood of $C$ contains the Frech\'{e}t-Urysohn fan
by Lemma~\ref{lem:fu}. 
We recall that 
there are infinitely many finite coverings of $M$, 
since $\pi_{1}(M)$ is residually finite. 
\end{proof}

\begin{lem}
\label{lem:fu}
Let $M,N \in \mathcal{H}$. 
Suppose that $f \colon N \to M$ is an $n$-sheeted covering for $n>1$. 
Then there are infinitely many Dehn fillings of $N$ 
which are not commensurable to any Dehn filling of $M$. 
\end{lem}
\begin{proof}
Let $N(r)$ be a Dehn filling of $N$ 
which is not compatible to $f$, i.e. 
the map $f$ cannot be extended to $N(r)$. 
For sufficiently long slopes $r$, 
the manifold $N(r)$ is not commensurable to any Dehn filling of $M$ 
by Lemma~\ref{lem:cover}. 
There are infinitely many such Dehn fillings. 
\end{proof}

We finally show that 
$\mathcal{C}$ and $\mathcal{C}_{\mathrm{orb}}$ have different topologies 
as stated in Remark~\ref{rem:orb}. 
Recall that 
$\mathcal{H}_{\mathrm{orb}}$ is the space of finite volume hyperbolic 3-orbifolds, 
and $\mathcal{C}_{\mathrm{orb}}$ is the quotient space of $\mathcal{H}_{\mathrm{orb}}$ 
up to commensurability. 

\begin{prop}
\label{prop:orb}
There is a convergent sequence in $\mathcal{C}_{\mathrm{orb}}$ 
which is divergent with respect to the topology of $\mathcal{C}$. 
\end{prop}
\begin{proof}
Let $M$ be a finite volume hyperbolic 3-manifold with a cusp. 
Fix a slope in a cusp of $M$. 
For a positive integer $n$, 
let $M_{n}$ denote the Dehn filling of $M$ along $n$ times of the slope, 
i.e. $M_{n}$ is a topologically homeomorphic to the Dehn filling $M_{1}$, 
and has a singular locus of order $n$ at the filled core. 
Then the orbifold version of hyperbolic Dehn surgery theorem states that 
$M_{n}$ is hyperbolic for sufficiently large $n$, 
and the sequence $\{ M_{n} \}$ converges to $M$ in $\mathcal{H}_{\mathrm{orb}}$. 
Hence $\{ \pi (M_{n}) \}$ converges to $\pi (M)$ in $\mathcal{C}_{\mathrm{orb}}$.

For sufficiently large $n$, the orbifold $M_{n}$ is non-arithmetic. 
Consider a manifold $N_{n} \in \mathcal{H}$ commensurable to $M_{n}$. 
By Theorem~\ref{thm:nonarithcomm}, 
there is $\hat{M}_{n} \in \mathcal{H}_{\mathrm{orb}}$ 
such that $M_{n}$ and $N_{n}$ are finite coverings of $\hat{M}_{n}$. 
Then $\vol (\hat{M}_{n}) \geq v_{0}$, 
where $v_{0} > 0$ is the smallest volume of a hyperbolic 3-orbifold 
as referred in the proof of Lemma~\ref{lem:cover}. 
Since the orbifold $\hat{M}_{n}$ is the image of a covering from $M_{n}$, 
it has a singular locus of order at least $n$. 
Hence the manifold covering $N_{n} \to \hat{M}_{n}$ has degree at least $n$. 
Then we obtain $\vol (N_{n}) \geq n \cdot \vol (\hat{M}_{n}) \geq n v_{0}$. 
Therefore the sequence $\{ \vol (N_{n}) \}$ is divergent.

We show that 
the sequence $\{ \pi (M_{n}) \} = \{ \pi (N_{n}) \}$ 
does not converge to $C = \pi (M)$ with respect to the topology of $\mathcal{C}$. 
Let $A = \bigcup_{n \geq n_{0}} \pi^{-1}(\pi (N_{n})) \subset \mathcal{H}$ for sufficiently large $n_{0}$. 
Then $C$ is disjoint from $\pi (A)$. 
As shown in the proof of Proposition~\ref{prop:t1}, 
$\pi^{-1}(\pi (N_{n}))$ is discrete for each $n$. 
Since $\{ \vol (N_{n}) \}$ is divergent for any choice of $N_{n}$, 
the subset $A$ is discrete. 
It is sufficient to construct a neighborhood $\pi (U)$ of $C$ in $\mathcal{C}$ 
disjoint from $\pi (A)$, 
where $\pi (U) = \pi (U(\{ \epsilon_{i,j} \}))$ was given in the proof of Theorem~\ref{thm:0dim}. 
Let $\pi^{-1}(C) = \{ M_{0,j} \} \subset \mathcal{H}$. 
At first, we can choose $\epsilon_{0,j} > 0$ 
such that $ U_{0,j} = B(M_{0,j}; \epsilon_{0,j})$ is disjoint from $A$, 
since $A$ is discrete. 
Hence $U_{0} = \bigcup_{j} U_{0,j}$ 
and $\pi^{-1}(\pi(U_{0})) = \{ M_{0,j} \}$ are disjoint from $A$. 
Inductively, 
we can take $\epsilon_{i,j} > 0$ 
such that $U_{i,j} = B(M_{i,j}; \epsilon_{i,j})$ is disjoint from $A$, 
where $U_{i-1} = U_{i-2} \cup \bigcup_{j} U_{i-1,j}$ and $\pi^{-1}(\pi(U_{i-1})) = \{ M_{i,j} \}$. 
Then $U = \bigcup_{i} U_{i}$ is disjoint from $A$, 
and $\pi (U)$ is disjoint from $\pi (A)$. 
\end{proof}

\section*{Acknowledgements} 

The author would like to thank Tomotada Ohtsuki for helpful discussions 
in the conference ``Intelligence of Low-dimensional Topology'', 
held at Research Institute for Mathematical Sciences, 
an International Joint Usage/Research Center located in Kyoto University. 
This work was supported by 
JSPS KAKENHI Grant Numbers 15H05739, 19K14530, 
and JST CREST Grant Number JPMJCR17J4.

\bibliography{ref-qtcc}

\end{document}